\documentclass[11pt]{amsart}
\usepackage[T1]{fontenc}
\usepackage[utf8]{inputenc}
\usepackage{lmodern}
\usepackage{amsmath,amssymb,amsthm,mathtools}
\usepackage{booktabs,array}
\usepackage{tikz-cd}
\usepackage[margin=28mm]{geometry}
\usepackage[expansion=false]{microtype}
\usepackage[hidelinks]{hyperref}
\hypersetup{pdftitle={A counterexample to the TF-cone conjecture},
 pdfauthor={Jindong Yan and Shenglin Zhu}}
\newtheorem{maintheorem}{Theorem}

\newtheorem{theorem}{Theorem}[section]
\newtheorem{proposition}[theorem]{Proposition}
\newtheorem{lemma}[theorem]{Lemma}
\newtheorem{corollary}[theorem]{Corollary}
\newtheorem{conjecture}{Conjecture}
\theoremstyle{definition}

\theoremstyle{remark}

\DeclareMathOperator{\Hom}{Hom}
\DeclareMathOperator{\Ext}{Ext}
\DeclareMathOperator{\End}{End}
\DeclareMathOperator{\Coker}{Coker}
\DeclareMathOperator{\Ker}{Ker}
\DeclareMathOperator{\im}{im}
\DeclareMathOperator{\rank}{rank}
\DeclareMathOperator{\ind}{ind}
\DeclareMathOperator{\Cone}{Cone}
\DeclareMathOperator{\pd}{pd}

\DeclareMathOperator{\Mor}{Mor}
\DeclareMathOperator{\rad}{rad}

\newcommand{\R}{\mathbb R}
\newcommand{\N}{\mathbb N}
\newcommand{\Kb}{K^{\mathrm b}}
\newcommand{\Db}{D^{\mathrm b}}
\allowdisplaybreaks[1]
\title{A counterexample to the TF-cone conjecture}
\author{Jindong Yan}
\thanks{\raggedright\setlength{\parindent}{\normalparindent} Jindong Yan (corresponding author):
School of Mathematical Sciences, Fudan University, Shanghai~200433, China.
Email: \href{mailto:22110180048@m.fudan.edu.cn}{\texttt{22110180048@m.fudan.edu.cn}}.}
\author{Shenglin Zhu}
\thanks{\raggedright\setlength{\parindent}{\normalparindent} Shenglin Zhu:
School of Mathematical Sciences, Fudan University, Shanghai~200433, China.
Email: \href{mailto:mazhusl@fudan.edu.cn}{\texttt{mazhusl@fudan.edu.cn}}.}
\date{}

\begin{document}
\begin{abstract}
The TF-cone conjecture of Asai and Iyama predicts that the TF
equivalence class of an integral $g$-vector is the open cone generated
by the canonical summands of all its positive multiples. We give a
counterexample to this conjecture. For the tensor
product of the three-arrow Kronecker algebra with the path algebra
of $s\to t$, a vector $w$ has a two-dimensional TF equivalence class,
whereas every positive integral multiple $mw$ is generically
indecomposable. The cone generated by these canonical summands is
therefore a single ray. The same example disproves the
relative-interior version of this conjecture proposed by Haerizadeh
and Yassemi.
\end{abstract}

\maketitle
\enlargethispage{4pt}

\section{Introduction and preliminaries}

For an integral vector $g$, Asai and Iyama conjecture that its TF
equivalence class is determined by its canonical decomposition
\cite[Conjecture~1.2]{AI}.  Canonical decomposition
of dimension vectors for acyclic quivers was developed by Kac and Schofield
\cite{Kac,Schofield}: the summand dimensions of a general
representation depend only on its dimension vector, and Schofield
gave an algorithm to compute them from the Euler form.  For
quivers with relations, representation varieties can have several
irreducible components.  Derksen and Fei instead considered
presentations between fixed projective modules, whose parameter
spaces are affine spaces, and extended canonical decomposition to
their classes in the projective Grothendieck group \cite{DF}.

Stability supplies a second geometry on this group.  King's
numerical notion of semistability identifies the modules selected
by geometric invariant theory and gives moduli spaces of
semistable representations \cite{King}.  To group real weights
that determine the same semistable torsion pairs, Asai introduced TF
equivalence \cite{Asai}.  The connection with
$\tau$-tilting theory is concrete: Br\"ustle, Smith and Treffinger
proved that the positive $g$-vector cone of a $\tau$-tilting pair
is a chamber \cite[Proposition~3.15]{BST}, and Asai proved that
every chamber arises in this way from a two-term silting complex
\cite[Theorem~1.4]{Asai}.  Moreover, the positive cone of
the indecomposable summands of a two-term presilting complex is
a TF equivalence class \cite[Theorem~1.3]{Asai}.  Using TF equivalence
classes, Asai gave a proof that the closed $g$-vector cones of two-term
silting complexes cover the real Grothendieck group if and only if
the algebra is $\tau$-tilting finite
\cite[Theorem~1.6]{Asai}.  Describing TF classes beyond these
cones is therefore a natural problem in the infinite case
\cite{AI,HY}.

Let $k$ be an algebraically closed field and $\Lambda$ a
finite-dimensional basic $k$-algebra.  For an integral weight
$g\in K_0(\operatorname{proj}\Lambda)$,
write $\ind(g)$ for the set of classes in the canonical decomposition
of $g$, and put
\[
 \ind(\N g)=\bigcup_{m\geq1}\ind(mg).
\]
Following the cone convention in \cite[Introduction]{AI}, for a
subset $S$ of a finite-dimensional real vector space set
\[
 \Cone(S)=
 \left\{\sum_{i=1}^r a_iv_i:
 r\geq0,\ a_i\in\R_{\geq0},\ v_i\in S\right\},
 \qquad
 \Cone^\circ(S)=\operatorname{relint}\Cone(S).
\]
Relative interior is taken in the linear span of the cone.
For a finite set $S=\{v_1,\ldots,v_r\}$, this convention gives
$\Cone^\circ(S)=\sum_{i=1}^r\R_{>0}v_i$.
When $S=\varnothing$, both cones are $\{0\}$.
Asai and Iyama proved the general inclusion
\cite[Theorem~1.1]{AI}
\[
 \Cone^\circ\bigl(\ind(g)\bigr)\subseteq[g]_{\mathrm{TF}},
\]
and equality holds for hereditary and $E$-tame algebras
\cite[Theorem~1.3]{AI}.  Since these algebras satisfy the ray
condition \cite[Propositions~6.7(a) and~7.1(a)]{AI}, the cones
generated by $\ind(g)$ and $\ind(\N g)$ coincide
\cite[Proposition~5.3(b)]{AI}.  The ray condition need not hold
in general \cite[Example~5.9(c)]{AI}.
Positive rescaling preserves TF equivalence, so including the
summands of all multiples leads to
the following conjecture.

\begin{conjecture}[Asai and Iyama
  {\cite[Conjecture~1.2]{AI}}]\label{conj:asai-iyama}
For every finite-dimensional basic $k$-algebra $\Lambda$ and every
$g\in K_0(\operatorname{proj}\Lambda)$,
\[
 [g]_{\mathrm{TF}}=\Cone^\circ\bigl(\ind(\N g)\bigr).
\]
\end{conjecture}

Haerizadeh and Yassemi study these cones as a way to describe
regions outside the chambers \cite[Introduction]{HY}.
For their open cone \cite[Definition~5.5]{HY}, we use the separate
notation
\[
 \mathcal C^\circ(g)
 :=\bigcup_{m\geq1}\Cone^\circ\bigl(\ind(mg)\bigr).
\]
Their relative-interior formulation is the following.

\begin{conjecture}[Haerizadeh and Yassemi
  {\cite[Conjecture~1.1]{HY}}]\label{conj:haerizadeh-yassemi}
For every finite-dimensional basic $k$-algebra $\Lambda$ and every
$g\in K_0(\operatorname{proj}\Lambda)$,
\[
 [g]_{\mathrm{TF}}^\circ=\mathcal C^\circ(g),
 \qquad
 [g]_{\mathrm{TF}}^\circ
 :=\operatorname{relint}[g]_{\mathrm{TF}},
\]
where relative interior is taken in the real linear span of the
TF class.
\end{conjecture}

Haerizadeh and Yassemi prove that the equality in
Conjecture~\ref{conj:haerizadeh-yassemi} is equivalent to
equality of the dimensions of the TF class and the cone generated
by $\ind(\N g)$ \cite[Theorem~1.4]{HY}.

We give a counterexample in which
\[
 \dim_{\mathbb{R}}\langle[g]_{\mathrm{TF}}\rangle_{\mathbb{R}}=2
 \quad\text{(Corollary~\ref{prop:tf-class})},
 \qquad
 \dim_{\mathbb{R}}\langle\mathcal{C}(g)\rangle_{\mathbb{R}}=1
 \quad\text{(Proposition~\ref{prop:generic-ind})}.
\]
Thus both conjectures fail.

Consider the three-arrow Kronecker quiver
\[
 Q_B:\quad
 \begin{tikzcd}[column sep=large]
  1 \arrow[r,bend left=45,"\alpha_1"]
    \arrow[r,"\alpha_2" description]
    \arrow[r,bend right=45,"\alpha_3"'] & 2
 \end{tikzcd}
 \qquad B=kQ_B.
\]
Let $Q_A$ be the quiver
\[
 \begin{tikzcd}[column sep=huge,row sep=huge]
  (1,s) \arrow[r,bend left=25,"\alpha_{1,s}"]
        \arrow[r,"\alpha_{2,s}" description]
        \arrow[r,bend right=25,"\alpha_{3,s}"']
        \arrow[d,"v_1"']
   & (2,s) \arrow[d,"v_2"] \\
  (1,t) \arrow[r,bend left=25,"\alpha_{1,t}"]
        \arrow[r,"\alpha_{2,t}" description]
        \arrow[r,bend right=25,"\alpha_{3,t}"']
   & (2,t)
 \end{tikzcd}
\]
Impose the three relations
\[
 v_2\alpha_{1,s}=\alpha_{1,t}v_1,\qquad
 v_2\alpha_{2,s}=\alpha_{2,t}v_1,\qquad
 v_2\alpha_{3,s}=\alpha_{3,t}v_1.
\]
If $I$ is the ideal generated by the differences of the two sides,
then
\[
 A=kQ_A/I\cong B\otimes_k k(s\longrightarrow t).
\]
Thus $A$ is a finite-dimensional basic algebra.
Let $P_1,P_2$ be the indecomposable projective $B$-modules at the two
vertices, and let $P_{i,s},P_{i,t}$ be the projective $A$-modules at
the corresponding vertices.  Define
\[
 h=[P_1]-2[P_2],\qquad
 p=[P_{1,s}]-2[P_{2,s}],\qquad
 q=[P_{1,t}]-2[P_{2,t}],\qquad
 w=p+q.
\]
Then we prove the following theorem, which yields a counterexample
to both conjectures.

\begin{maintheorem}\label{thm:counterexample}
For the algebra $A$ and the weight $w$ above,
\[
 [w]_{\mathrm{TF}}
 =[w]_{\mathrm{TF}}^\circ
 =\R_{>0}p+\R_{>0}q.
\]
For every positive integer $m$, the canonical decomposition of $mw$
has one summand:
\[
 \ind(mw)=\{mw\}.
\]
Consequently,
\[
 \Cone^\circ\bigl(\ind(\N w)\bigr)
 =\mathcal C^\circ(w)=\R_{>0}w
 \ \subsetneq\ [w]_{\mathrm{TF}}^\circ.
\]
\end{maintheorem}

In this paper, we proceed as follows.
In Section~2, we compute the TF equivalence class of $w$ in
Corollary~\ref{prop:tf-class}.
In Section~\ref{sec:schur-resolution}, we construct, for each positive
integer $m$, a brick $M_m$ whose minimal projective resolution
has $g$-vector $mw$.
In Section~4, we use these resolutions and upper semicontinuity of
endomorphism dimensions to prove generic indecomposability of $mw$
in Proposition~\ref{prop:generic-ind}, completing the proof of
Theorem~\ref{thm:counterexample}.
In Section~5, we compute the directed $E$-invariants of $mp$ and $nq$.

We recall the definitions and notation used in the proofs.

\subsection{Weights and TF equivalence}

Consider the categories $\operatorname{mod}\Lambda$ and
$\operatorname{proj}\Lambda$ of finite-dimensional left modules and
finitely generated projective modules, respectively.
For a module $N$ and an integer $r\geq0$,
$N^r$ denotes the direct sum of $r$ copies of $N$.  We write
\[
 K_0(\operatorname{proj}\Lambda)_\R
 =K_0(\operatorname{proj}\Lambda)\otimes_\mathbb Z\R.
\]
The pairing
\[
 \langle[P],[M]\rangle=\dim_k\Hom_\Lambda(P,M)
 \qquad(P\in\operatorname{proj}\Lambda,\ M\in\operatorname{mod}\Lambda)
\]
extends linearly to real classes.  For a weight
$\theta\in K_0(\operatorname{proj}\Lambda)_\R$, we abbreviate
$\theta(M)=\langle\theta,[M]\rangle$.

The two non-strict semistable classes are
\[
 \begin{aligned}
 \overline{\mathcal T}_\theta
 &=\{M:\theta(N)\geq0
          \text{ for every quotient }N\text{ of }M\},\\
 \overline{\mathcal F}_\theta
 &=\{M:\theta(L)\leq0
          \text{ for every submodule }L\text{ of }M\}.
 \end{aligned}
\]
Their intersection
$\mathcal W_\theta=\overline{\mathcal T}_\theta
\cap\overline{\mathcal F}_\theta$ is the category of
$\theta$-semistable modules.  Equivalently, $M$ is semistable when
$\theta(M)=0$ and $\theta(L)\leq0$ for every submodule $L\subseteq M$.
Two weights $\theta,\eta\in K_0(\operatorname{proj}\Lambda)_\R$
are TF equivalent when their two non-strict classes coincide.
For $\theta\in K_0(\operatorname{proj}\Lambda)_\R$, its TF
equivalence class is
\[
 [\theta]_{\mathrm{TF}}
 =\{\eta\in K_0(\operatorname{proj}\Lambda)_\R:
   \overline{\mathcal T}_\eta=\overline{\mathcal T}_\theta,\
   \overline{\mathcal F}_\eta=\overline{\mathcal F}_\theta\}.
\]

\subsection{Presentations and canonical decomposition}

Choose representatives $P_1,\ldots,P_n$ of the indecomposable
projective $\Lambda$-modules.  For
$g=\sum_i c_i[P_i]\in K_0(\operatorname{proj}\Lambda)$, define
\[
 P_g^+=\bigoplus_{c_i>0}P_i^{c_i},\qquad
 P_g^-=\bigoplus_{c_i<0}P_i^{-c_i}.
\]
The source and target have no common indecomposable projective
summand.  The normalized presentation space of $g$ is the affine space
\[
 \mathcal H_\Lambda(g)=\Hom_\Lambda(P_g^-,P_g^+).
\]
A property of a general presentation means a property holding
on a nonempty Zariski open subset, such a subset being dense in
$\mathcal H_\Lambda(g)$.
To $d\in\mathcal H_\Lambda(g)$ we associate the two-term complex
\[
 P_d=[P_g^-\xrightarrow{d}P_g^+]
 \quad\text{in degrees }-1,0,
 \qquad [P_d]=[P_g^+]-[P_g^-]=g.
\]
The bounded homotopy category of projective modules is denoted by
$\Kb(\operatorname{proj}\Lambda)$, and the bounded derived category
by $\Db(\operatorname{mod}\Lambda)$.

For $g\in K_0(\operatorname{proj}\Lambda)$, it is generically
indecomposable if a general presentation in $\mathcal H_\Lambda(g)$
is indecomposable as a two-term complex.
The canonical decomposition
$g=g_1\oplus\cdots\oplus g_r$ records the classes of the
indecomposable summands of a general presentation:
\[
 P_d\simeq P_{d_1}\oplus\cdots\oplus P_{d_r},
 \qquad d_i\in\mathcal H_\Lambda(g_i).
\]
It exists and is unique up to reordering
\cite[Definition~4.3 and Theorem~4.4]{DF};
see also \cite[Definition~2.18 and Proposition~2.21]{AI}.
Multiplicities are retained in this decomposition, whereas
$\ind(g)=\{g_1,\ldots,g_r\}$ is a set.  We set $\ind(0)=\varnothing$.
For a module $M$ with a minimal projective presentation
$P^-\to P^+\to M\to0$, its $g$-vector is
$g(M)=[P^+]-[P^-]$.

\subsection{Modules over the tensor product}

\begin{lemma}\label{lem:arrow-category}
There is an equivalence of categories
\[
 \operatorname{mod}A\simeq\Mor(\operatorname{mod}B).
\]
It identifies an $A$-module with a morphism of $B$-modules
$M=(X\xrightarrow{f}Y)$, with spaces $X_i$ and $Y_i$ at
$(i,s)$ and $(i,t)$.  A morphism from $(X\xrightarrow{f}Y)$ to
$(X'\xrightarrow{f'}Y')$ is a pair of $B$-module maps $(u,v)$
satisfying $vf=f'u$.
\end{lemma}

\begin{proof}
The horizontal arrows of an $A$-representation give two
$B$-modules $X$ and $Y$.  The three relations say exactly that
the vertical maps $f_i:X_i\to Y_i$ commute with all three
horizontal arrows, so they form a $B$-module map $f:X\to Y$.
Conversely, every such map defines an $A$-representation
satisfying the relations.  The vertexwise conditions on a map of
representations say that its two horizontal components are
$B$-module maps and that $vf=f'u$.  These constructions are
inverse on objects and morphisms.
\end{proof}

For $M=(X\xrightarrow{f}Y)$, we order its dimension vector as
\[
 \dim M=(\dim_kX_1,\dim_kX_2,\dim_kY_1,\dim_kY_2),
\]
corresponding to $((1,s),(2,s),(1,t),(2,t))$.

For $N\in\operatorname{mod}B$, put
\[
 L_sN=(N\xrightarrow{I_N}N),\qquad L_tN=(0\longrightarrow N).
\]
Maps from these modules are determined by their components at $s$
and $t$, respectively, giving
\[
 \Hom_A(L_sN,M)\simeq\Hom_B(N,X),\qquad
 \Hom_A(L_tN,M)\simeq\Hom_B(N,Y).
\]
The projective $B$-modules have dimension vectors
$\dim P_1=(1,3)$ and $\dim P_2=(0,1)$, and satisfy
$\Hom_B(P_i,N)\simeq N_i$.  Thus
\[
 P_{i,s}=L_sP_i,\qquad P_{i,t}=L_tP_i
\]
are the four indecomposable projective $A$-modules, with
\[
 \Hom_A(P_{i,s},M)\simeq X_i,\qquad
 \Hom_A(P_{i,t},M)\simeq Y_i.
\]
If $S_i$ is the simple $B$-module at vertex $i$, the simple
$A$-modules are
$S_{i,s}=(S_i\to0)$ and $S_{i,t}=(0\to S_i)$.
In particular, for
$\theta=\sum_i(a_{i,s}[P_{i,s}]+a_{i,t}[P_{i,t}])$ one has
\[
 \theta(M)=\sum_i
 \bigl(a_{i,s}\dim_kX_i+a_{i,t}\dim_kY_i\bigr).
\]
The weights used in the theorem therefore satisfy
\[
 h(N)=\dim_kN_1-2\dim_kN_2,\qquad
 (ap+bq)(X\xrightarrow{f}Y)=a\,h(X)+b\,h(Y)
 \quad(a,b\in\R).
\]
We call a module a brick when its endomorphism algebra is $k$.

\section{The TF equivalence class}

We compute the full TF equivalence class of $w$, proving the first
assertion of Theorem~\ref{thm:counterexample}.  The key step is to
express the semistable torsion and torsion-free classes over $A$
in terms of their counterparts over $B$.

For a full subcategory $\mathcal C\subseteq\operatorname{mod}B$, let
$\Mor(\mathcal C)$ denote the full subcategory of $\operatorname{mod}A$
whose objects are morphisms $f\colon X\to Y$ with $X,Y\in\mathcal C$.
The class $\overline{\mathcal T}_h$ is closed under quotients
and extensions, and $\overline{\mathcal F}_h$ under submodules and
extensions \cite[Subsection~2.1 and Definition~2.4(a)]{AI}.

\begin{lemma}\label{lem:torsion-morphism}
For all real numbers $a,b>0$, one has
\[
 \overline{\mathcal F}_{ap+bq}=\Mor(\overline{\mathcal F}_h),
 \qquad
 \overline{\mathcal T}_{ap+bq}=\Mor(\overline{\mathcal T}_h),
 \qquad
 \mathcal W_{ap+bq}=\Mor(\mathcal W_h).
\]
\end{lemma}

\begin{proof}
Write $\theta=ap+bq$ and $M=(X\xrightarrow{f}Y)$.  A subobject of
$M$ is a pair of $B$-submodules $X'\subseteq X$, $Y'\subseteq Y$
with $f(X')\subseteq Y'$, together with the restricted morphism
$X'\to Y'$.  Its quotient is
$(X/X'\to Y/Y')$, with the morphism induced by $f$.

Suppose that $M\in\overline{\mathcal F}_\theta$.
For every $B$-submodule $Y'\subseteq Y$, the object $(0\to Y')$
is a subobject of $M$.  Hence $b h(Y')\leq0$, and positivity of $b$
gives $Y\in\overline{\mathcal F}_h$.
Every submodule $K\subseteq\Ker f$ gives a subobject $(K\to0)$,
so $a h(K)\leq0$ and $\Ker f\in\overline{\mathcal F}_h$.
For any $X'\subseteq X$, the exact sequence of $B$-modules
\[
 0\longrightarrow X'\cap\Ker f\longrightarrow X'
   \xrightarrow{\,f\,} f(X')\longrightarrow0
\]
has its left term in $\overline{\mathcal F}_h$ and its right term
a submodule of $Y$.  Additivity therefore gives
\[
 h(X')=h(X'\cap\Ker f)+h(f(X'))\leq0.
\]
Thus $X\in\overline{\mathcal F}_h$.
Conversely, if $X,Y\in\overline{\mathcal F}_h$, then every
subobject $(X'\to Y')$ of $M$ has weight
$a h(X')+b h(Y')\leq0$.  This proves the first equality.

Suppose now that $M\in\overline{\mathcal T}_\theta$.
For each $X'\subseteq X$, there is an epimorphism
\[
 (X\xrightarrow{f}Y)\longrightarrow(X/X'\to0)
\]
whose component maps are the quotient $X\to X/X'$ and the zero
map $Y\to0$; its kernel is $(X'\to Y)$.
Consequently $a h(X/X')\geq0$ for every $X'$, so
$X\in\overline{\mathcal T}_h$.
Set $C=\Coker f$.  Each quotient $C\twoheadrightarrow C'$ yields
an epimorphism $M\to(0\to C')$, whose target component is the
composite $Y\to C\to C'$.  This composite annihilates $\im f$,
so the two component maps form a morphism of $A$-modules.
Its weight gives $b h(C')\geq0$ for every quotient $C'$ of $C$,
and hence $C\in\overline{\mathcal T}_h$.

To recover the condition on $Y$, let $Y'\subseteq Y$ be arbitrary.
The exact sequence
\[
 0\longrightarrow(\im f+Y')/Y'
 \longrightarrow Y/Y'
 \longrightarrow Y/(\im f+Y')\longrightarrow0
\]
has left term a quotient of $X$ and right term a quotient of $C$.
Indeed, the map $X\to(\im f+Y')/Y'$ induced by $f$ is surjective,
and $Y\to Y/(\im f+Y')$ factors through $C$.
Both outer terms have nonnegative $h$-weight, so
$h(Y/Y')\geq0$.  Therefore $Y\in\overline{\mathcal T}_h$.
Conversely, if $X,Y\in\overline{\mathcal T}_h$, every quotient
$(X/X'\to Y/Y')$ of $M$ has nonnegative $\theta$-weight.
This proves the second equality.

Intersecting the first two equalities gives
\[
 \mathcal W_\theta
 =\Mor(\overline{\mathcal F}_h\cap\overline{\mathcal T}_h)
 =\Mor(\mathcal W_h).
\]
\end{proof}

\begin{corollary}\label{prop:tf-class}
The full TF equivalence class of $w$ is
\[
 [w]_{\mathrm{TF}}
 =\{ap+bq\mid a,b>0\}
 =\operatorname{relint}(\mathbb R_{\geq0}p+\mathbb R_{\geq0}q).
\]
It has dimension $2$, its closure is
$\mathbb R_{\geq0}p+\mathbb R_{\geq0}q$, and
\[
 \mathcal W_w=\Mor(\mathcal W_h).
\]
\end{corollary}

\begin{proof}
Lemma~\ref{lem:torsion-morphism}, applied also to $a=b=1$, shows
that every $ap+bq$ with $a,b>0$ is TF equivalent to $w$ and gives
the stated formula for $\mathcal W_w$.
Conversely, let $\eta\in[w]_{\mathrm{TF}}$ and write
\[
 \eta=x_s[P_{1,s}]+y_s[P_{2,s}]
      +x_t[P_{1,t}]+y_t[P_{2,t}].
\]
Take the $B$-module $T=(k^2,k)$ with arrow maps
$(1,0)$, $(0,1)$, and $0$.  Their common kernel is zero, so every
nonzero submodule $U=(U_1,U_2)$ has $U_2=k$ and
$h(U)=\dim_kU_1-2\leq0$.  Since $h(T)=0$, we have
$T\in\mathcal W_h$.  Thus $(T\to0)$ and $(0\to T)$ belong to
$\mathcal W_w=\mathcal W_\eta$, and their $\eta$-weights give
\[
 2x_s+y_s=0,\qquad 2x_t+y_t=0.
\]
It follows that $\eta=ap+bq$ with $a=x_s$ and $b=x_t$.

The simples $S_{1,s}$ and $S_{1,t}$ have $w$-weight $1$, so they
belong to $\overline{\mathcal T}_w$ but not to
$\overline{\mathcal F}_w$.  TF equivalence preserves these
conditions, giving $a=\eta(S_{1,s})>0$ and
$b=\eta(S_{1,t})>0$.

The vectors $p$ and $q$ are linearly independent, so the remaining
assertions follow from the dimension, closure, and relative
interior of the positive quadrant in their span.
\end{proof}

\section{Bricks and minimal projective resolutions}
\label{sec:schur-resolution}

We construct a brick with $g$-vector $mw$ for every positive integer
$m$.  Direct calculations of maps and extensions to the simple
modules determine the terms of its minimal projective resolution.

\subsection{A uniform family of bricks}

Fix $m\geq 1$.  Choose pairwise distinct elements
$d_1,\ldots,d_m\in k$, and put
\[
 D=\operatorname{diag}(d_1,\ldots,d_m),
 \qquad Se_i=e_{i+1},
\]
where the subscripts are taken modulo $m$.  Since $k$ is infinite, we
can choose $\lambda\in k$ outside the finite set
$\{d_i-d_j:1\leq i,j\leq m\}$.  The matrix
$D'=D+\lambda I_m$ then has pairwise distinct eigenvalues, and its
spectrum is disjoint from that of $D$.  Define $B$-modules by listing
their three arrow maps:
\begin{equation}\label{eq:schur-family}
 X=(I_m,D,S),\qquad
 Y=\left(I_{2m},
       \begin{pmatrix}D&0\\0&D'\end{pmatrix},
       \begin{pmatrix}S&I_m\\0&S\end{pmatrix}\right).
\end{equation}
At each of the two vertices, the map
\[
 J=\begin{pmatrix}I_m\\0\end{pmatrix}:k^m\longrightarrow k^{2m}
\]
intertwines the corresponding arrow maps of $X$ and $Y$.  It therefore
defines the $A$-module
\[
 M_m=(X\xrightarrow{J}Y),\qquad
 \dim M_m=(m,m,2m,2m).
\]

\begin{proposition}\label{prop:schur}
For every $m\geq1$, the module $M_m$ is a brick:
$\End_A(M_m)=k$.
\end{proposition}

\begin{proof}
The first arrow of $Y$ is the identity, so an endomorphism of $Y$ has
the same matrix $F$ at both vertices.  Commutation with its second
arrow forces $F$ to be diagonal, because all $2m$ diagonal entries of
$\operatorname{diag}(D,D')$ are distinct.  Write
$F=\operatorname{diag}(C_1,C_2)$, where $C_1$ and $C_2$ are diagonal
$m\times m$ matrices.  Commutation with the third arrow gives
\[
 C_1S=SC_1,\qquad C_2S=SC_2,\qquad C_1=C_2.
\]
A diagonal matrix commuting with the cyclic permutation $S$ has all
its diagonal entries equal: applying the commutation relation to
$e_i$ identifies the entries at $i$ and $i+1$.  Thus
$C_1=C_2=cI_m$, and $\End_B(Y)=k$.

An endomorphism of $M_m$ is a pair $(U,V)$ with
$U\in\End_B(X)$, $V\in\End_B(Y)$, and $VJ=JU$.  The preceding
calculation gives $V=cI_Y$.  Since $J$ is injective at both vertices,
the equality $JU=cJ$ implies $U=cI_X$.  Every scalar gives such a
pair, proving the assertion.
\end{proof}

The quotient of $Y$ by the image of $J$ is
$Z=(I_m,D',S)$.  The embedding and quotient therefore give a short
exact sequence in $\operatorname{mod}A$:
\begin{equation}\label{eq:schur-extension}
 0\longrightarrow L_sX
 \xrightarrow{(I_X,J)} M_m
 \longrightarrow L_tZ\longrightarrow0.
\end{equation}
\subsection{Hom and Ext to the simples}

Let $N$ be a $B$-module of dimension vector $(r,r)$, where $r\geq0$,
whose first arrow is $I_r$ in chosen bases.  A morphism from $N$ to
$S_1=(k\to0)$ consists of an arbitrary linear map $k^r\to k$ at
vertex $1$ and the zero map at vertex $2$; all three arrow
conditions are satisfied.  Hence
$\dim_k\Hom_B(N,S_1)=r$.  A morphism to $S_2=(0\to k)$ consists
of a map $v:k^r\to k$ at vertex $2$.  The first arrow condition
is $vI_r=0$, so $v=0$ and $\Hom_B(N,S_2)=0$.

The acyclic path algebra $B$ is hereditary, and its Euler identity
for finite-dimensional modules $U,V$ is
\[
 \begin{aligned}
 \dim_k\Hom_B(U,V)-\dim_k\Ext_B^1(U,V)
 &=\langle\dim U,\dim V\rangle_B,\\
 \langle(x_1,x_2),(y_1,y_2)\rangle_B
 &=x_1y_1+x_2y_2-3x_1y_2.
 \end{aligned}
\]
The vertex and arrow contributions give
\[
 \langle(r,r),(1,0)\rangle_B=r,
 \qquad
 \langle(r,r),(0,1)\rangle_B=r-3r=-2r.
\]
Subtracting these values from the Hom dimensions and using
heredity gives
\begin{equation}\label{eq:schur-b-simple-ext}
\begin{array}{c|cc}
 &S_1&S_2\\ \hline
 \dim_k\Hom_B(N,-)&r&0\\
 \dim_k\Ext_B^1(N,-)&0&2r
\end{array}
\qquad
 \Ext_B^j(N,S_i)=0\quad(j\geq2).
\end{equation}
In particular, this calculation applies to $X$ and $Z$, with $r=m$.

For a $B$-module $N$ and an $A$-module
$R=(R_s\xrightarrow{\rho}R_t)$, the evaluation adjunctions are
\begin{equation}\label{eq:schur-adjunction}
 \Hom_A(L_sN,R)\cong\Hom_B(N,R_s),\qquad
 \Hom_A(L_tN,R)\cong\Hom_B(N,R_t).
\end{equation}
A morphism from $(N\xrightarrow{I_N}N)$ is determined by its
$s$-component $u$, with $t$-component $\rho u$, and a morphism from
$(0\to N)$ is determined by its $t$-component.  Both evaluation
functors are exact.  The adjunctions therefore show that $L_s$ and
$L_t$ send projective $B$-modules to projective $A$-modules.
The functors $L_s$ and $L_t$ are themselves exact, so applying
either functor to a projective resolution of $N$ gives a projective
resolution over $A$.  Using \eqref{eq:schur-adjunction} in each
degree yields, for every $j\geq0$,
\begin{equation}\label{eq:schur-derived-adjunction}
 \Ext_A^j(L_sN,R)\cong\Ext_B^j(N,R_s),\qquad
 \Ext_A^j(L_tN,R)\cong\Ext_B^j(N,R_t),
\end{equation}
where $\Ext^0$ denotes $\Hom$.

The $t$-component of $S_{i,s}$ and the $s$-component of $S_{i,t}$
are zero.  Thus \eqref{eq:schur-derived-adjunction} gives
\[
 \Ext_A^j(L_tZ,S_{i,s})=0,
 \qquad
 \Ext_A^j(L_sX,S_{i,t})=0
 \quad(j\geq0).
\]
Applying the long exact Ext sequence to
\eqref{eq:schur-extension}, these vanishings give
\begin{equation}\label{eq:schur-simple-adjunction}
 \Ext_A^j(M_m,S_{i,s})\cong\Ext_B^j(X,S_i),\qquad
 \Ext_A^j(M_m,S_{i,t})\cong\Ext_B^j(Z,S_i)
 \quad(j\geq0).
\end{equation}
Together with \eqref{eq:schur-b-simple-ext}, these isomorphisms give
\begin{equation}\label{eq:schur-simple-table}
\begin{array}{c|cccc}
 &S_{1,s}&S_{2,s}&S_{1,t}&S_{2,t}\\ \hline
 \dim_k\Hom_A(M_m,-)&m&0&m&0\\
 \dim_k\Ext_A^1(M_m,-)&0&2m&0&2m\\
 \dim_k\Ext_A^2(M_m,-)&0&0&0&0
\end{array}.
\end{equation}
All higher groups $\Ext_A^j(M_m,S_{i,\epsilon})$ vanish as well,
for $j\geq2$ and $\epsilon\in\{s,t\}$.

Take a minimal projective resolution
\[
 \cdots\longrightarrow Q_2\longrightarrow Q_1
 \longrightarrow Q_0\longrightarrow M_m\longrightarrow0,
 \qquad
 Q_j=\bigoplus_v P_v^{\oplus\beta_{j,v}},
\]
where $v$ runs through the four vertices of $A$.  Minimality means
that the image of $Q_{j+1}\to Q_j$ is contained in $\rad Q_j$.
Every map from $Q_j$ to a simple module annihilates $\rad Q_j$,
so all differentials in $\Hom_A(Q_\bullet,S_v)$ are zero.
Since $A$ is split basic, $\dim_k\Hom_A(P_u,S_v)=\delta_{uv}$,
and hence
\begin{equation}\label{eq:schur-betti}
 \beta_{j,v}=\dim_k\Hom_A(Q_j,S_v)
 =\dim_k\Ext_A^j(M_m,S_v)
 \quad(j\geq0).
\end{equation}
The zero row for $\Ext_A^2$ in \eqref{eq:schur-simple-table}
gives $Q_2=0$.  Exactness of the resolution then makes
$Q_1\to Q_0$ injective.  The rows for $\Hom$ and $\Ext_A^1$
identify its two terms as
\begin{equation}\label{eq:schur-projective-terms}
 P_m^+=P_{1,s}^{\oplus m}\oplus P_{1,t}^{\oplus m},\qquad
 P_m^-=P_{2,s}^{\oplus 2m}\oplus P_{2,t}^{\oplus 2m}.
\end{equation}
Since $m\geq1$, the $\Ext_A^1$ row is nonzero, and hence
$\pd_A M_m=1$.

The resulting minimal projective resolution is
\begin{equation}\label{eq:schur-minimal-resolution}
 0\longrightarrow P_{2,s}^{2m}\oplus P_{2,t}^{2m}
 \xrightarrow{d_m}P_{1,s}^{m}\oplus P_{1,t}^{m}
 \xrightarrow{\varepsilon_m}M_m\longrightarrow0.
\end{equation}
Consequently,
\[
 g(M_m)=[P_m^+]-[P_m^-]=m(p+q)=mw.
\]

\section{Generic indecomposability of \texorpdfstring{$mw$}{mw}}

Fix $m\geq1$.  With the projective terms of
\eqref{eq:schur-projective-terms}, put
\[
 \begin{gathered}
 P_m^-=P_{2,s}^{2m}\oplus P_{2,t}^{2m},\qquad
 P_m^+=P_{1,s}^{m}\oplus P_{1,t}^{m},\\
 \mathcal H_m=\Hom_A(P_m^-,P_m^+).
 \end{gathered}
\]
Thus $\mathcal H_m=\mathcal H_A(mw)$ is the normalized
presentation space for $mw$.

\begin{proposition}\label{prop:generic-ind}
There is a nonempty dense open subset $\mathcal V_m$ of
$\mathcal H_m$ such that every $d\in\mathcal V_m$ is injective,
its cokernel is a brick, and $d$ is indecomposable as a
projective presentation.  Consequently,
\[
 \ind(mw)=\{mw\}\qquad\text{for every }m\geq1.
\]
\end{proposition}

\begin{proof}
Let
\[
 \mathcal O_m=\{d\in\mathcal H_m:d\text{ is injective}\}.
\]
The function $d\mapsto\dim_k\ker d$ is upper semicontinuous on
$\mathcal H_m$, so $\mathcal O_m$ is open.  It is nonempty,
since it contains the differential $d_m$ of
\eqref{eq:schur-minimal-resolution}, whose cokernel is $M_m$.

On $\mathcal O_m$, the cokernels form an algebraic family of
modules with fixed dimension vector $(m,m,2m,2m)$.  The function
$d\mapsto\dim_k\End_A(\Coker d)$ is upper semicontinuous
\cite[Lemma~4.3]{CBS}.  Put
\[
 \mathcal V_m=\{d\in\mathcal O_m:
   \dim_k\End_A(\Coker d)=1\}.
\]
This set is nonempty, since $\Coker d_m=M_m$ is a brick by
Proposition~\ref{prop:schur}.  All these cokernels are nonzero,
so their endomorphism dimensions are at least one.
Upper semicontinuity therefore makes $\mathcal V_m$ open in
$\mathcal O_m$, and hence in $\mathcal H_m$.
It is dense because $\mathcal H_m$ is an irreducible affine space.

For every $d\in\mathcal V_m$, the exact sequence
\[
 0\longrightarrow P_m^-\xrightarrow{d}P_m^+
 \longrightarrow\Coker d\longrightarrow0
\]
is a minimal projective resolution, since $P_m^-$ and $P_m^+$
have no common indecomposable summand.  Its cokernel is a brick
and hence indecomposable, so its minimal projective resolution
$P_d$ is indecomposable.  Thus a general presentation in
$\mathcal H_A(mw)$ is indecomposable, giving
$\ind(mw)=\{mw\}$.
\end{proof}

\begin{proof}[Proof of Theorem~\ref{thm:counterexample}]
Corollary~\ref{prop:tf-class} gives
\[
 [w]_{\mathrm{TF}}=\R_{>0}p+\R_{>0}q.
\]
The vectors $p,q$ are linearly independent, so this open
quadrant equals its relative interior
$[w]_{\mathrm{TF}}^\circ$.
By Proposition~\ref{prop:generic-ind},
$\ind(mw)=\{mw\}$ for every $m\geq1$, and hence
\[
 \ind(\N w)=\{mw:m\geq1\},\qquad
 \Cone^\circ\bigl(\ind(\N w)\bigr)
 =\mathcal C^\circ(w)=\R_{>0}w.
\]
Indeed, the positive multiples $mw$ all generate the same
ray, and $\Cone^\circ(\ind(mw))=\R_{>0}w$ for each $m$.
Since $w=p+q$, this ray lies in the positive quadrant.
The inclusion is strict: $2p+q$ belongs to
$\R_{>0}p+\R_{>0}q$ and is not a scalar multiple of $p+q$.
\end{proof}

\section{Directed \texorpdfstring{$E$}{E}-invariants}

For integral weights $g,h$, the directed invariant is
\cite[Definition~2.19]{AI}
\[
 E_A(g,h)=
 \min_{\substack{f\in\mathcal H_A(g)\\d\in\mathcal H_A(h)}}
 \dim_k\Hom_{\Kb(\operatorname{proj}A)}(P_f,P_d[1]).
\]

\begin{proposition}\label{prop:cross-e}
For all positive integers $m,n$,
\[
 E_A(mp,nq)=0,\qquad E_A(nq,mp)=mn.
\]
\end{proposition}

\begin{proof}
Take arbitrary $f\in\mathcal H_A(mp)$ and $d\in\mathcal H_A(nq)$,
with associated complexes
\[
 P_f=[P_{2,s}^{2m}\xrightarrow{f}P_{1,s}^{m}],\qquad
 P_d=[P_{2,t}^{2n}\xrightarrow{d}P_{1,t}^{n}].
\]
Evaluation at $s$ gives $\Hom_A(P_{2,s},P_{1,t})=0$.
The degree-one term of $\Hom_A^\bullet(P_f,P_d)$ is therefore
zero, so $E_A(mp,nq)=0$.

For the reverse direction, let $C^\bullet=\Hom_A^\bullet(P_d,P_f)$.
Evaluation at $t$ gives
\[
 \begin{aligned}
 C^{-1}&=\Hom_A(P_{1,t}^{n},P_{2,s}^{2m})=0,\\
 C^0&=\Hom_A(P_{2,t}^{2n},P_{2,s}^{2m})
       \oplus\Hom_A(P_{1,t}^{n},P_{1,s}^{m}),\\
 C^1&=\Hom_A(P_{2,t}^{2n},P_{1,s}^{m}).
 \end{aligned}
\]
Thus $\dim_k C^0=4mn+mn=5mn$ and $\dim_k C^1=6mn$.
Since $C^2=0$, for every pair $(d,f)$ one has
\[
 \dim_k\Hom_{\Kb(\operatorname{proj}A)}(P_d,P_f[1])
 =\dim_kH^1(C^\bullet)
 =6mn-\rank\partial^0\geq mn.
\]

To attain this bound, choose $\lambda\ne\mu$ in $k$ and set
\[
 N_1=(I_m,\lambda I_m,0),\qquad
 N_2=(I_n,\mu I_n,0).
\]
A morphism $N_2\to N_1$ has the same matrix $T$ at both vertices
and satisfies $\lambda T=\mu T$, so $\Hom_B(N_2,N_1)=0$.
The Euler identity gives
\[
 \dim_k\Ext_B^1(N_2,N_1)
 =-\langle(n,n),(m,m)\rangle_B=mn.
\]
By \eqref{eq:schur-b-simple-ext}, their minimal projective
resolutions have terms $P_2^{2m}\to P_1^m$ and
$P_2^{2n}\to P_1^n$, respectively.
Applying $L_s$ to the first and $L_t$ to the second gives
points $f\in\mathcal H_A(mp)$ and $d\in\mathcal H_A(nq)$
whose complexes resolve $L_sN_1$ and $L_tN_2$.
The derived adjunction \eqref{eq:schur-derived-adjunction} yields
\[
 \Hom_{\Kb(\operatorname{proj}A)}(P_d,P_f[1])
 \cong\Ext_A^1(L_tN_2,L_sN_1)
 \cong\Ext_B^1(N_2,N_1).
\]
Its dimension is $mn$, proving $E_A(nq,mp)=mn$.
\end{proof}

\end{document}